\documentclass[letterpaper, 10 pt, conference]{ieeeconf}
\IEEEoverridecommandlockouts 
\usepackage{float}
\usepackage[utf8]{inputenc}
\usepackage[hidelinks]{hyperref}

\usepackage{amsmath,amssymb,amsthm}
\usepackage{graphicx}
\usepackage{enumerate}
\usepackage[hidelinks]{hyperref}
\usepackage{algorithm}
\usepackage{algpseudocode}
\usepackage{tikz}
\usetikzlibrary{arrows, calc, backgrounds, shapes.misc, positioning, shapes, fit}
\usepackage[belowskip=-0.15in,aboveskip=0.05in,font=small]{caption}
\usepackage{svg}

\newcommand{\figref}[1]{Fig.~\ref{#1}}

\newcommand{\bmtx}{\begin{bmatrix}}
\newcommand{\emtx}{\end{bmatrix}}
\newcommand{\bsmtx}{\left[ \begin{smallmatrix}} 
\newcommand{\esmtx}{\end{smallmatrix} \right]}

\newcommand{\field}[1]{\mathbb{#1}}
\newcommand{\R}{\field{R}}
\newcommand{\Sm}{\field{S}}

\newtheorem{defn}{Definition}

\newtheorem{theorem}{Theorem}
\newtheorem{rem}{Remark}

\tikzset{>=stealth'}
\colorlet{dgreen}{black!60!green}
\colorlet{lgreen}{green!10!white}
\colorlet{axis}{gray!80!white}

\title{\LARGE \bf Robust Control Barrier Functions with Sector-Bounded Uncertainties}
\author{Jyot Buch, Shih-Chi Liao, Peter Seiler
  \thanks{This work was funded by the US ONR grant N00014-18-1-2209.}
  \thanks{Jyot Buch is with the Department of Aerospace Engineering and Mechanics at the University of Minnesota, Twin Cities, {\tt\small buch0271@umn.edu}}
  \thanks{Shih-Chi Liao and Peter Seiler are with the Electrical Engineering and Computer Science department at the University of Michigan, Ann Arbor, {\tt\small \{shihchil,pseiler\}@umich.edu}}}

\begin{document}
	\setlength{\textfloatsep}{0.2in}
	\setlength{\floatsep}{0.2in}
\maketitle
\thispagestyle{empty}
\pagestyle{empty}
\tikzset{every picture/.style={thick,scale=1},every node/.style={scale=1}}

\begin{abstract}
This paper focuses on safety critical control with sector-bounded uncertainties at the plant input. The uncertainties can represent nonlinear and/or time-varying components. We propose a new robust control barrier function (RCBF) approach to enforce safety requirements in the presence of these uncertainties. The primary objective is to minimally alter the given baseline control command to guarantee safety in the presence of modeled uncertainty. The resulting min-norm optimization problem can be recast as a Second-Order Cone Program (SOCP) to enable online implementation. Properties of this controller are studied and a numerical example is provided to illustrate the effectiveness of this approach.
\end{abstract}

\section{Introduction}
\label{intro}

This paper presents safety critical control for a plant with sector-bounded input uncertainties. There are many applications including autonomous driving, medical or industrial robotics, and aerospace vehicles that require prioritizing safety over performance objectives~\cite{knight2002safety}. One of the popular methods to encode safety is by means of Control Barrier Functions (CBF), which can be used as a constraint in a quadratic program to modify control actions to adhere to safety specifications~\cite{ames2016control,ames2019control}. CBFs can be designed using a representative or a surrogate model of the system~\cite{ames2019control} or can be learned online~\cite{taylor2020learning,choi2020reinforcement}. Often, for simplicity, actuator nonlinearities are ignored, which raises robustness concerns. 

Our approach is to characterize the input-output behavior of nonlinearities at the plant input using point-wise in time quadratic constraints. We then use the robust control barrier functions (RCBFs) presented in Section~\ref{sec:RCBF} to provide safety guarantees for the entire uncertainty set. This is done by ensuring that a safe action always exists for all nonlinearites at the modeled uncertainty level. This combines a traditional robust control approach with the CBF methods for safety critical control. As a result, we obtain more cautious trajectories when close to the unsafe region. 

There are three main contributions of the paper. First, we present a new robust control barrier function based approach to handle sector-bounded uncertainties at the plant input. This allows us to handle nonlinearities and time-varying memoryless uncertainties described by a quadratic constraint. Second, we formulate an optimization problem that minimally alters the control command to guarantee safety in the presence of modeled input uncertainty. This optimization problem can be rewritten in terms of a second-order cone program (SOCP) to be solved online. Finally, the proposed approach is demonstrated using a lateral vehicle control example to study robust safety.

There is a large body of literature on CBFs with a good overview provided in~\cite{ames2019control}. Only the most closely related work is summarized here. Robust control barrier functions are presented for guaranteeing safety in the presence of $\mathcal{L}_\infty$ bounded disturbances in~\cite{xu2015robustness,garg2021robust,breeden2021robust} and stochastic disturbances in~\cite{takano2018application}. The work in~\cite{nguyen2021robust} also considers robust CBFs to account for the changes in the dynamics as a perturbation to the vector field. A key distinction is that the input nonlinearities in our work lead to uncertainties that depend on the control decisions, which is not allowed in the framework of~\cite{xu2015robustness,garg2021robust,breeden2021robust,nguyen2021robust}. In this paper, we provide safety guarantees for static nonlinearities and/or time-varying memoryless uncertainties at the plant input. The most recent work in~\cite{pete2021arxiv} considers a more general class of unmodeled dynamics (e.g. unknown time-delays, actuation lag, etc.) using $\alpha$-IQCs and CBFs. The price for this generality is that the trajectories tend to be even more conservative than those obtained with the method presented here. Other related work on robust CBFs includes~\cite{jankovic2018robust,choi2021robust,dean2020guaranteeing}.

\noindent\textbf{Notation:} Let $\R^{n \times m}$ and $\Sm^{n}$ denote the sets of $n$-by-$m$ real matrices and $n$-by-$n$ real, symmetric matrices. The Euclidean norm of a vector $\mathbf{v}\in\R^m$ is defined as $\|\mathbf{v}\|_2 := \sqrt{\mathbf{v}^\top \mathbf{v}}$. A continuous function $\eta :\R\rightarrow\R$ is called extended class-$\mathcal{K}_{\infty}$ $(\mathcal{K}_{\infty,e})$ if it is strictly monotonically increasing and satisfies $\eta(0) = 0, \lim_{r\rightarrow -\infty} \eta(r) = -\infty$, and $\lim_{r\rightarrow \infty} \eta(r) = \infty$. 

\section{Preliminaries}
\label{sec:prelim}

\subsection{Problem Formulation}
\label{sec:problem}
Consider the design interconnection as shown in~\figref{fig:designic}. The uncertain plant $P$ is described as a series interconnection of known part $G$ and an unknown perturbation $\phi$ at the plant input. 
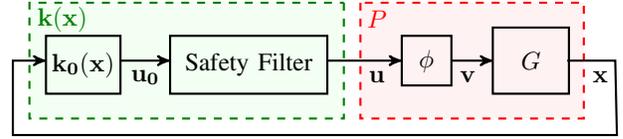
\begin{figure}
	\centering
	\begin{tikzpicture}[thick,scale=1.1,every node/.style={scale=0.95}]
	\draw [dashed,green!50!black, fill=green!5!white] (-4.2,1.1) rectangle (-8,-0.3);
	\draw [dashed,red,fill=red!5!white] (-4,1.1) rectangle (-1.3,-0.3);
	\draw (-2.4,0) rectangle node{$G$}(-1.48,0.8);
    \node at (-6.6,0.2) {$\mathbf{u_0}$};
    \node at (-3.8,0.2) {$\mathbf{u}$};
    \node at (-1.1,0.2) {$\mathbf{x}$};
    \node at (-2.7,0.2) {$\mathbf{v}$};
    \draw (-7.8,0.7) rectangle node{$\mathbf{k_0}(\mathbf{x})$}(-6.9,0);
    \draw (-6.3,0.7) rectangle node{Safety Filter}(-4.4,0);
    \draw (-3.5,0.7) rectangle node{$\phi$}(-2.9,0.1);
    \draw [->](-1.5,0.4) -- (-0.92,0.4) -- (-0.9,-0.5) -- (-8.2,-0.5) -- (-8.2,0.4) -- (-7.8,0.4);
    \draw [->](-6.9,0.4) -- (-6.3,0.4);
    \draw [->](-2.9,0.4) -- (-2.4,0.4);
    \draw [->](-4.4,0.4) -- (-3.5,0.4);
    \node [red] at (-3.8,0.9) {$P$};
    \node [green!50!black] at (-7.6,0.9) {$\mathbf{k(x)}$};
	\end{tikzpicture}
	\caption{Uncertain State-Feedback Design Interconnection}
	\label{fig:designic}
\end{figure}
This perturbation represents nonlinearities and/or time-varying, memoryless uncertainties. We will refer to $\phi$ as a ``nonlinearity" for simplicity. Let $P$ be given with the following input-affine dynamics:
\begin{align}
    \label{eq:system}
    \begin{split}
        \dot{\mathbf{x}}(t) &= \mathbf{f}(\mathbf{x}(t)) + \mathbf{g}(\mathbf{x}(t)) \,\mathbf{v}(t), \hspace{0.2in} \mathbf{x}(0) = \mathbf{x}_0\\
	    \mathbf{v}(t) &= \phi(\mathbf{u}(t),t)
    \end{split}
\end{align}
where $\mathbf{x}(t) \in D\subset \R^{n}$ is the state, $\mathbf{u}(t) \in \mathcal{U} \subset \R^{m}$ is the admissible control input and $\mathbf{v}(t) \in \mathcal{V}(\mathbf{u}(t)) \subset \R^{m}$ is the uncertain input. Moreover, $\mathbf{f}:D\subset \R^n \rightarrow \R^n$ and $\mathbf{g}:D\subset\R^n \rightarrow \R^{n\times m}$ are locally Lipschitz continuous functions of the state $\mathbf{x}$. It is assumed that the dynamics given by~\eqref{eq:system} are defined on open set $D\subset \R^{n}$ and are forward complete, i.e. for every initial condition $\mathbf{x}(0) \in D$, there exists a unique solution $\mathbf{x}(t)$ for all $t\geq 0$. The nonlinearity $\phi$ is assumed to lie in a sector $\left[\alpha,\beta\right]$ with $0 < \alpha \leq 1 \leq\beta$ so that the sector-bound contains the nominal case $\mathbf{v}=\mathbf{u}$. This sector-bound can be written as the following point-wise in time quadratic constraint (Section 6.1 of~\cite{khalil2002nonlinear}):
\begin{align}
    \label{eq:sectorqc}
    \left[\mathbf{v}(t)-\alpha \mathbf{u}(t)\right]^\top \left[\beta \mathbf{u}(t)-\mathbf{v}(t)\right]\geq 0,\,\, \forall t \geq 0.
\end{align}
If $m=1$ then the single control channel sector-bound can be illustrated as shown in~\figref{fig:sectorbnd}. The uncertainty $\phi$ lies in a sector $\left[\alpha,\beta\right]$ when it can be bounded by two lines with slopes of $\alpha$ and $\beta$, respectively. The shaded gray region represents the allowable uncertainty set. In more general setting when $m\neq 1$, the constraint~\eqref{eq:sectorqc} allows cross-coupling between the input channels. 
\begin{figure}
	\centering
	\begin{tikzpicture}[thick,scale=1,rounded corners = 0.5mm,every node/.style={scale=0.9}]
	\draw [->,axis](2.9,-1) -- (7.1,-1);
	\draw [->,axis](5,-2.8) -- (5,0.9);
	\draw[white,fill=gray!15!white] (6.9,0.1) -- ++(-0.8,0.7) -- ++(-1.1,-1.8);
	\draw[white,fill=gray!15!white] (3.9,-2.8) -- ++(1.1,1.8) -- ++(-1.9,-1.1);
	\draw [-,black](6.1,0.8) -- (3.9,-2.8);
	\draw [-,black](6.9,0.1) -- (3.1,-2.1);
	\node at (7,-1.2) {$u$};
	\node at (4.8,0.8) {$v$};
	\node [black] at (7.1,-0.2) {$\alpha u $};
	\node [black] at (6,1) {$\beta u$};
	\node [red!90!black] at (7.1,0.6) {$v = \phi(u,t)$};
	\draw [red!90!black,very thick] plot[smooth, tension=.7] coordinates { (3.4,-2.4) (3.9,-2.2) (4.1,-1.7) (4.7,-1.3) (5.3,-0.7) (5.7,-0.3) (6.1,0.2) (6.7,0.4)};
	\end{tikzpicture}
	\caption{Illustration of SISO Nonlinearity $\phi \in \left[\alpha,\beta\right]$}
	\label{fig:sectorbnd}
\end{figure}
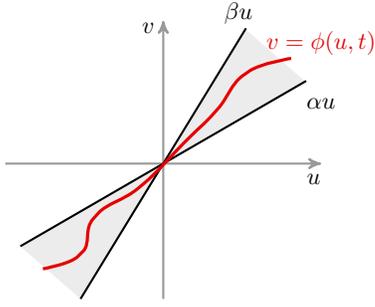

We assume that a locally Lipschitz continuous function $\mathbf{k_0}:D\subset\R^n\rightarrow\mathcal{U} \subset\R^m$ is given such that the baseline (not necessarily safe) control law is $\mathbf{u_0} = \mathbf{k_0}(\mathbf{x})$. The notion of safety is formalized by defining a safe set $\mathcal{C} \subset D \subset \R^n$ in the state space that the system must remain within~\cite{ames2016control}. In particular, consider the set $\mathcal{C}$ as the zero-superlevel set of a continuously differentiable function $h : D\subset\R^n \rightarrow \R$:
\begin{align}
\label{eq:safeset}
    \mathcal{C} &\triangleq \{ \mathbf{x} \in D \subset \R^n : h(\mathbf{x}) \geq 0\}
\end{align}
The boundary and interior of the safe set are denoted as $\partial\mathcal{C}$ and $\mbox{Int}(\mathcal{C})$, respectively. It is assumed that zero is a regular value of $h$ and $\mathcal{C}$ is non-empty with no isolated points. These assumptions mean that $h(\mathbf{x}) = 0$ implies $\frac{\partial h}{\partial \mathbf{x}} (\mathbf{x}) \neq 0$, Int$(\mathcal{C}) \neq \emptyset$, and $\overline{\mbox{Int}(\mathcal{C})} = \mathcal{C}$. Explicit time dependence of variables are omitted when it is clear from the context. It is assumed that the initial condition is in the safe set i.e. $\mathbf{x}_0 \in \mathcal{C}$. 

Our primary goal is to design a safety filter in~\figref{fig:designic} that minimally alters the baseline control command $\mathbf{u_0}$ so the state of the closed-loop system remains safe even in the presence of the nonlinearity. The proposed solution in Section~\ref{sec:optcontrol} is an optimization problem that can be solved online to compute safe control action $\mathbf{u}\in\mathcal{U}$.

\subsection{Background}
\label{sec:back}

This section provides a brief summary on set-invariance using control barrier functions~\cite{ames2016control}. First, consider the nominal case i.e. without nonlinearities at the plant input. In this case, $\mathbf{v} = \mathbf{u}$. If we let $\mathbf{u} = \mathbf{k(x)}$ then the closed-loop dynamics are given by:
\begin{align}
\label{eq:clsys}
     \dot{\mathbf{x}} = \mathbf{f_{cl}}(\mathbf{x}) = \mathbf{f}(\mathbf{x}) + \mathbf{g}(\mathbf{x}) \,\mathbf{k}(\mathbf{x}), \hspace{0.2in} \mathbf{x}(0) = \mathbf{x}_0
\end{align}
In the context of the autonomous system above, safety is synonymous with the forward invariance of~$\mathcal{C}$~\cite{ames2016control}:
\begin{defn}[Forward Invariance and Safety]
 A set $\mathcal{C} \subset D \subset \R^n$ is forward invariant if for every initial condition $\mathbf{x_0} \in \mathcal{C}$, the solution to the closed-loop system~\eqref{eq:clsys} satisfies $\mathbf{x}(t) \in \mathcal{C}$ for all $t \geq 0$. The system~\eqref{eq:clsys} is safe with respect to $\mathcal{C}$ if $\mathcal{C}$ is forward invariant.
\end{defn}
As mentioned in the previous section, the baseline control law $\mathbf{u_0}=\mathbf{k_0(x)}$ is not necessarily safe. In this case, we would like to ensure that a safe action $\mathbf{u}$ exists which can steer the system to remain within the safe set $\mathcal{C}$. This is formalized by defining the notion of control invariance.
\begin{defn}[Control Invariance]
A set $\mathcal{C}$ is control invariant if there exists a controller $\mathbf{k}:D\subset\R^n\rightarrow\mathcal{U}\subset\R^m$ such that $\mathcal{C}$ is forward invariant with respect to the system~\eqref{eq:clsys}.
\end{defn}
Control barrier functions are used to ensure control invariance for the nominal plant. Let $L_\mathbf{f}h:= \frac{\partial h}{\partial \mathbf{x}} \mathbf{f}$ and $L_\mathbf{g}h:=\frac{\partial h}{\partial \mathbf{x}} \mathbf{g}$ denote the Lie derivatives of $h$ with respect to $\mathbf{f}$ and $\mathbf{g}$. 
\begin{defn}[Control Barrier Function]
 Let $\mathcal{C} \subset D \subset \R^n$ be the zero-superlevel set of a continuously differentiable function $h : D\subset\R^n \rightarrow \R$. Then $h$ is a Control Barrier Function if there exists $\eta \in \mathcal{K}_{\infty,e}$ such that for all $\mathbf{x}\in D$:
\begin{align}
    \label{eq:CBF}
    \sup_{\mathbf{u}\in \mathcal{U}}
    \left[ L_\mathbf{f}h(\mathbf{x}) + L_\mathbf{g}h(\mathbf{x})\, \mathbf{u}\right] \geq -\eta(h(\mathbf{x}))
\end{align}
\end{defn}
The existence of a control barrier function $h$ satisfying constraint~\eqref{eq:CBF} implies that if the state reaches the boundary of~$\mathcal{C}$, then the control input $\mathbf{u}$ can be used to prevent the state from entering the unsafe region. The main result in~\cite{ames2016control,ames2019control} can be used to design a Lipschitz continuous safety-filter that yields safety for the nominal closed-loop. The specific implementation involves solving the following quadratic program online with CBF constraint to minimally alter the baseline action $\mathbf{u_0}$. 
\begin{align}
    \tag{CBF-QP}
    \label{eq:CBFQP}
    \mathbf{u}^*(\mathbf{x}) =\, &\arg \min_{\mathbf{u}\in\mathcal{U}} \frac{1}{2} \|\mathbf{u} - \mathbf{u_0}\|^2_2\\
    &\mbox{s.t. }  L_\mathbf{f}h(\mathbf{x}) + L_\mathbf{g}h(\mathbf{x})\, \mathbf{u} \geq -\eta(h(\mathbf{x}))\nonumber
\end{align}
In general, constructing a CBF $h$ is not straightforward and often requires careful consideration~\cite{ames2019control}. 

Next, consider the uncertain case with the sector-bounded nonlinearity $\phi$ at the plant input. This nonlinearity can alter the control command resulting in a safety violation. The uncertain closed-loop dynamics with the nonlinearity can be written as follows:
\begin{align}
\label{eq:uclsys}
     \dot{\mathbf{x}} = \mathbf{f_{ucl}}(\mathbf{x}) = \mathbf{f}(\mathbf{x}) + \mathbf{g}(\mathbf{x}) \,\phi(\mathbf{k}(\mathbf{x}),t), \hspace{0.1in} \mathbf{x}(0) = \mathbf{x}_0
\end{align}
The notion of robust control invariance~(Definition 4.4 in~\cite{blanchini2015set}) is useful to guarantee safety in the presence of uncertainties. A specific definition for the sector-bounded nonlinearity is given next.
\begin{defn}[Robust Control Invariance and Robust Safety]
A set $\mathcal{C}$ is robust control invariant if there exists a controller $\mathbf{k}:D\subset\R^n\rightarrow\mathcal{U}\subset\R^m$, such that $\mathcal{C}$ is forward invariant with respect to the uncertain closed-loop~\eqref{eq:uclsys} for all nonlinearities $\phi$ in a sector $[\alpha,\beta]$. The system~\eqref{eq:uclsys} is robustly safe with respect to $\mathcal{C}$ if $\mathcal{C}$ is robust control invariant.
\end{defn}
Robust control invariance can be ensured by showing the existence of a robust control barrier function as presented in the following section.

\section{Main Results}
\label{sec:mainres}

\subsection{Uncertainty Mapping}
\label{sec:uncmapping}

The first step is to use a loop-shifting transformation (Section 6.5 of~\cite{khalil2002nonlinear}) to map the nonlinearity $\phi \in [\alpha,\beta]$ into a normalized, input additive form, as in standard robust control workflow. To make this precise, define $\Delta:\R^m\times\R \rightarrow \R^m$ such that $\mathbf{v}(t)=\phi(\mathbf{u}(t),t)$ is mapped to:
\begin{align}
   \mathbf{v}(t) = \frac{1}{2}\, (\alpha+\beta) (\mathbf{u}(t) + \Delta(\mathbf{u}(t),t))
\end{align}
The mapped nonlinearity satisfies $\Delta \in [-\theta, +\theta]$ where $\theta:= (\beta-\alpha)/(\beta+\alpha)$. This re-centers the sector-bound to $0$ and separates the nominal control action $\mathbf{u}(t)$ and the uncertain control command $\Delta(\mathbf{u}(t),t)$. The factor $\frac{1}{2}\, (\alpha+\beta)$ scales the input function $\mathbf{g}(\mathbf{x})$ to be $\mathbf{\tilde{g}}(\mathbf{x}) := \frac{1}{2}\,(\alpha+\beta)\,\mathbf{g}(\mathbf{x})$. This yields the following input-affine system with mapped nonlinearity:
\begin{align}
\label{eq:mappedsys}
\begin{split}
\dot{\mathbf{x}}(t) &= \mathbf{f}(\mathbf{x}(t)) + \mathbf{\tilde{g}}(\mathbf{x}(t)) \,(\mathbf{u}(t) + \mathbf{w}(t)), \hspace{0.1in} \mathbf{x}(0) = \mathbf{x}_0\\
    \mathbf{w}(t) &= \Delta(\mathbf{u}(t),t)
\end{split}
\end{align}
Assume the uncertainty level satisfies $0\leq\theta<1$. The symmetric sector constraint on the mapped nonlinearity $\Delta$ corresponds to a norm bound:
\begin{align}
    \label{eq:normbnd}
    \|\mathbf{w}(t)\|_2 \leq \theta \|\mathbf{u}(t)\|_2, \,\, \forall t \geq 0.
\end{align}
Note that as $\theta\rightarrow 0$, i.e. as $\alpha$ and $\beta$ both tend to $1$, then $\mathbf{w}(t) \rightarrow 0$ and we recover the nominal plant $G$. For simplicity, the remainder of the paper considers the system~\eqref{eq:mappedsys} with mapped nonlinearity $\Delta$ instead of the original system~\eqref{eq:system} with $\phi$.
Let $\mathcal{W}(\mathbf{u}(t))$ denote the set of uncertain inputs $\mathbf{w}(t)$ satisfying the norm bound constraint~\eqref{eq:normbnd}. The set $\mathcal{W}(\mathbf{u}(t))$ depends on the control input $\mathbf{u}$ at time $t$. Thus, the uncertain input $\mathbf{w}(t)$ can not be simply treated as an exogenous disturbance input, as it depends on $\mathbf{u}(t)$ through the nonlinearity $\Delta$.

\subsection{Robust Control Barrier Functions (RCBF)}
\label{sec:RCBF}
Robust control barrier functions defined in this section can be used to synthesize controllers ensuring the safety of the uncertain closed-loop system with respect to a given set $\mathcal{C}$. 
\begin{defn}
\label{def:RCBF}
Let $\mathcal{C} \subset D \subset \R^n$
be a safe set given by~\eqref{eq:safeset}. The function $h$ is a Robust Control
Barrier Function for~\eqref{eq:mappedsys} if there exists $\eta \in \mathcal{K}_{\infty,e}$ 
such that for all $x\in D$:
\begin{align}
    \label{eq:RCBF}
    \sup_{\mathbf{u}\in \mathcal{U}}\, \inf_{\mathbf{w}\in\mathcal{W}}  \left[ L_\mathbf{f}h(\mathbf{x}) + L_\mathbf{\tilde{g}}h(\mathbf{x}) (\mathbf{u} + \mathbf{w}) \right] \geq -\eta(h(\mathbf{x}))
\end{align}
\end{defn}
The nonlinearity $\Delta$ can, in the worst case, yield an uncertain input $\mathbf{w}$ that minimizes the left side of inequality~\eqref{eq:RCBF}. We aim to choose a single control input $\mathbf{u} \in\mathcal{U}$ for all uncertain inputs $\mathbf{w}\in\mathcal{W}$ satisfying the constraint~\eqref{eq:normbnd}. The worst-case uncertain input $\mathbf{w}^*(\mathbf{u})$ after solving the inner optimization problem is given by:
\begin{align}
    \label{eq:wcw}
    \mathbf{w}^*(\mathbf{u}) = - \theta \|\mathbf{u}\|_2 \frac{L_\mathbf{\tilde{g}}h(\mathbf{x})^\top}{\|L_\mathbf{\tilde{g}}h(\mathbf{x})\|_2}
\end{align}
This follows from the linear cost in~\eqref{eq:RCBF}, but a more formal argument using Lagrange relaxation is given in Appendix~\ref{sec:wcw}. It can be verified that as $\theta\rightarrow0$, we have $\mathbf{w}^*(\mathbf{u})\rightarrow0$. Plugging in for $\mathbf{w^*}(\mathbf{u})$ in condition~\eqref{eq:RCBF} yields:
\begin{align}
    \label{eq:RCBF_wc}
    \sup_{\mathbf{u}\in \mathcal{U}}\, \left[ L_\mathbf{f}h(\mathbf{x}) + L_\mathbf{\tilde{g}}h(\mathbf{x}) (\mathbf{u} + \mathbf{w^*}(\mathbf{u})) \right] \geq -\eta(h(\mathbf{x}))
\end{align}
Define $p(\mathbf{x}) := L_\mathbf{f}h(\mathbf{x}) + \eta(h(\mathbf{x}))$ and the set of all control actions that render the set $\mathcal{C}$ robustly safe as follows:
\begin{align}
    \mathcal{U}_{RCBF}(\mathbf{x}) := \{ \mathbf{u} \in \mathcal{U} : p(\mathbf{x}) + L_\mathbf{\tilde{g}}h(\mathbf{x}) (\mathbf{u} + \mathbf{w^*}(\mathbf{u}))\nonumber \geq 0 \}
\end{align}
The feasibility of RCBF constraint~\eqref{eq:RCBF} ensures that the above set is nonempty. This implies that if the state reaches the boundary of $\mathcal{C}$ then there exists a control input to prevent the state of the uncertain closed-loop from crossing out of the safe set. Thus, the existence of a robust control barrier function implies that the system is robustly safe. This statement is formalized in the next theorem, which can be viewed as a robust version of the main result in~\cite{ames2016control}.
\begin{theorem}
Let $\mathcal{C} \subset D$ be a safe set defined using~\eqref{eq:safeset} as the superlevel set of a continuously differentiable function $h : D\subset\R^n \rightarrow \R$. If $h$ is a robust control barrier function on $D$ and $\frac{\partial h}{\partial \mathbf{x}}(\mathbf{x}) \neq 0$ for all $\mathbf{x} \in \partial \mathcal{C}$, then any Lipschitz continuous controller $\mathbf{u}(\mathbf{x}) \in \mathcal{U}_{RCBF}(\mathbf{x})$ for the system~\eqref{eq:mappedsys} renders the set $\mathcal{C}$ robust control invariant.
\end{theorem}
\begin{proof} If $h$ is a RCBF on open set $D$ then for any $\mathbf{x}\in\partial\mathcal{C}$ and for all $\mathbf{w}\in\mathcal{W}$,  there exists a control input $\mathbf{u}\in\mathcal{U}$ such that 
$\dot{h}(\mathbf{x,w,u}) \geq \eta(h(\mathbf{x})) = 0$. For a Lipschitz continuous control law $\mathbf{u}(\mathbf{x})\in \mathcal{U}_{RCBF}(\mathbf{x})$, according to generalizations of Nagumo’s theorem (Theorem $4.10$ in~\cite{blanchini2015set}) the closed set $\mathcal{C}$ is robust control invariant.
\end{proof}



\subsection{Optimization-Based Control}
\label{sec:optcontrol}


We can solve the following optimization online to compute a control input $\mathbf{u}$ that minimally alters the baseline input $\mathbf{u}_0$ to ensure safety:
\begin{align}
    \label{eq:RCBFOpt}
    \mathbf{u}^*(\mathbf{x}) =\, &\arg \min_{\mathbf{u}\in\mathcal{U}} \frac{1}{2} \|\mathbf{u} - \mathbf{u_0}\|^2_2\\
    &\mbox{ s.t. } p(\mathbf{x}) + L_\mathbf{\tilde{g}}h(\mathbf{x}) (\mathbf{u} + \mathbf{w}^*(\mathbf{u})) \geq 0\nonumber
\end{align}
Plugging in for $\mathbf{w}^*(\mathbf{u})$ and expanding the cost function yields the following problem with an equivalent optimizer. 
\begin{align}
    \label{eq:RCBFOpt1}
    \mathbf{u}^*(\mathbf{x}) =\, &\arg \min_{\mathbf{u}\in\mathcal{U}} \left[ \frac{1}{2}\mathbf{u}^\top \mathbf{u} - \mathbf{u_0}^\top \mathbf{u} \right]\\
    &\mbox{ s.t. } p(\mathbf{x}) + L_\mathbf{\tilde{g}}h(\mathbf{x}) \mathbf{u} - \theta \|\mathbf{u}\|_2 \|L_\mathbf{\tilde{g}}h(\mathbf{x})\|_2 \geq 0\nonumber
\end{align}
Again, the last term drops out as the uncertainty level $\theta \to 0$, yielding an affine constraint in $\mathbf{u}$. In this case, the above optimization problem is simply a (nominal) \ref{eq:CBFQP}. However, for $\theta>0$, the decision variable $\mathbf{u}$ in the constraint of~\eqref{eq:RCBFOpt1} appears as the Euclidean norm $\|\mathbf{u}\|_2$. Reformulate the problem~\eqref{eq:RCBFOpt1} using a slack variable $q$ as follows:
\begin{align}
    \label{eq:RCBFOpt2}
    \bsmtx\mathbf{u}^*(\mathbf{x})\\ q^*(\mathbf{x}) \esmtx =\, &\arg \min_{ \mathbf{u}\in\mathcal{U},q} \left[q - \mathbf{u_0}^\top \mathbf{u} \right]\\
    &\mbox{ s.t. } p(\mathbf{x}) + L_\mathbf{\tilde{g}}h(\mathbf{x}) \mathbf{u} \geq \theta \|L_\mathbf{\tilde{g}}h(\mathbf{x})\|_2 \|\mathbf{u}\|_2 \nonumber\\
    &\hspace{0.2in} 2q \geq \|\mathbf{u}\|_2^2\nonumber 
\end{align}
This yields a minimization problem over  $\mathbf{u}\in\mathcal{U}$ and $q>0$. The optimal solutions are related by $2q^*={\mathbf{u}^*}^\top \mathbf{u}^*$. The second constraint can be rewritten as a rotated second-order cone (SOC) condition as in Section 10.1 of~\cite{calafiore2014optimization}, which yields the optimization~\eqref{eq:RCBFOpt2} as follows:
\begin{align}\tag{RCBF-SOCP}
    \label{eq:SOCP}
    \bsmtx\mathbf{u}^*(\mathbf{x})\\ q^*(\mathbf{x}) \esmtx =\, &\arg \min_{ \mathbf{u}\in\mathcal{U}, q} \left[q - \mathbf{u_0}^\top \mathbf{u} \right]\\
    &\mbox{ s.t. } \theta \|L_\mathbf{\tilde{g}}h(\mathbf{x})\|_2 \|\mathbf{u}\|_2\leq p(\mathbf{x}) + L_\mathbf{\tilde{g}}h(\mathbf{x}) \mathbf{u} \nonumber\\
    &\hspace{0.2in} \Big\|\bsmtx \sqrt{2}\,\mathbf{u}\\ q - 1 \esmtx\Big\|_2 \leq q + 1\nonumber 
\end{align}
This problem falls under a special class of convex optimization problems known as second-order cone programs (SOCP)~\cite{boyd2004convex}, which can be solved online using existing numerical solvers. At higher uncertainty levels, the~\ref{eq:SOCP} problem (if feasible) yields more cautious (conservative and safe) control actions to prevent the states from entering the unsafe region. For the point-wise feasibility of the~\ref{eq:SOCP}, it is assumed that the set of control inputs $\mathcal{U}$ is not overly restrictive, thus allowing us to have sufficient control authority to maintain safety in the presence of modeled uncertainty. However, an approach similar to~\cite{zeng2021safety} can also be used to relax this assumption.

\subsection{Lipschitz Continuity}
\label{sec:lipcont}

This section discusses a key Lipschitz continuity property of the~\ref{eq:SOCP} problem. In the nominal case ($\theta =0$), if $\mathcal{U}\equiv\R^m$, then the~\ref{eq:CBFQP} has only a single linear constraint in $\mathbf{u}$, and in this special case, there is an explicit solution. This solution can be used to show that the resulting safety filter is a locally Lipschitz continuous function of the state $\mathbf{x}\in D$ (Theorem~$8$ of~\cite{xu2015robustness}). For the robust case ($\theta \neq 0$), the optimization problem~\eqref{eq:RCBFOpt} can compactly be written as:
\begin{align}
    \label{eq:RCBFOpt3}
    \mathbf{u}^*(\mathbf{x}) =\, &\arg \min_{\mathbf{u}\in\mathcal{U}_{RCBF}(\mathbf{x})} \frac{1}{2} \|\mathbf{u} - \mathbf{u_0}\|^2_2
\end{align}
To the best of our knowledge, there is no explicit solution to this general problem. However, it is a standard projection problem over the parameterized non-empty closed convex set, i.e. the optimizer $\mathbf{u}^*(\mathbf{x})$ is a projection of $\mathbf{u}_0 = \mathbf{k}_0(\mathbf{x})$ onto the set $\mathcal{U}_{RCBF}(\mathbf{x})$. The main results in Section~$6$ of~\cite{bednarczuk2020lipschitz} show that $\mathbf{u^*(x)}$ is a locally Lipschitz continuous function of $\mathbf{x}$, if the set $\mathcal{U}_{RCBF}(\mathbf{x})$ is described by polyhedral constraints parameterized by the state $\mathbf{x}$. We conjecture that $\mathbf{u^*(x)}$ remains locally Lipschitz when $\mathcal{U}_{RCBF}(\mathbf{x})$ is described by the two SOC constraints as in the~\ref{eq:SOCP} problem. Future work will focus on investigating this further. 

The remainder of this section presents a Lipschitz continuity result for the special case of scalar control input $u\in\R$. For this case $m=1$ and $L_{\tilde{\mathbf{g}}} h(\mathbf{x})\in\R$. The optimization~\eqref{eq:RCBFOpt} can be written as follows:
\begin{align}
    \label{eq:sclarporb}
    u^*(\mathbf{x}) =\, &\arg \min_{u} \frac{1}{2} (u - u_0)^2 \\
    &\mbox{s.t. } p(\mathbf{x}) + L_{\tilde{\mathbf{g}}} h(\mathbf{x})\, u - \theta |u| |L_{\tilde{\mathbf{g}}}h(\mathbf{x})| \geq 0\nonumber
\end{align}
The next theorem makes a formal statement about the Lipschitz continuity of the function $u^*$. A direct and more concise (independent of~\cite{bednarczuk2020lipschitz}) proof is provided in this case.

\begin{theorem}
Assume $\mathbf{f}:D\subset \R^n \rightarrow \R^n$, $\mathbf{g}:D\subset \R^n \rightarrow \R^{n\times1}$, $\eta\in\mathcal{K}_{\infty,e}$ and $u_0$ are given locally Lipschitz continuous functions. Let $u\in\R$ and $h : D\subset \R^n \rightarrow \R$ be a locally Lipschitz continuous robust control barrier function. Moreover, let $L_{\tilde{\mathbf{g}}}h(\mathbf{x})\neq0$. Then the solution, $u^*(\mathbf{x})$ of~\eqref{eq:sclarporb} is a locally Lipschitz continuous function of $\mathbf{x}\in D$.
\end{theorem}
\begin{proof}
Assume $L_{\tilde{\mathbf{g}}}h(\mathbf{x}) > 0$ and $0\leq\theta<1$. The constraint from~\eqref{eq:sclarporb} can be written as:
\begin{align}
\label{eq:constraint1}
    u - \theta |u| \geq \frac{-p(\mathbf{x})}{L_{\tilde{\mathbf{g}}}h(\mathbf{x})} 
\end{align}
If $\frac{-p(\mathbf{x})}{L_{\tilde{\mathbf{g}}}h(\mathbf{x})}\geq 0$ then $u\geq0$. Hence the constraint~\eqref{eq:constraint1} is:
\begin{align}
\label{eq:constraint2}
    u \geq \frac{-p(\mathbf{x})}{(1-\theta)L_{\tilde{\mathbf{g}}}h(\mathbf{x})} 
\end{align}
If $\frac{-p(\mathbf{x})}{L_{\tilde{\mathbf{g}}}h(\mathbf{x})}< 0$ then~\eqref{eq:constraint1} is satisfied for any $u\geq0$. Moreover it is satisfied for negative values of $u$ that satisfy:
\begin{align}
\label{eq:constraint3}
    u \geq \frac{-p(\mathbf{x})}{(1+\theta)L_{\tilde{\mathbf{g}}}h(\mathbf{x})} 
\end{align}
Thus the constraint in~\eqref{eq:sclarporb} is a parameterized interval of the form $u(\mathbf{x})\in [u_l(\mathbf{x}),+\infty)$ where $u_l(\mathbf{x})$ is given by:
\begin{align}
  u_l(\mathbf{x}) = \max\left\{
  \frac{-p(\mathbf{x})}{(1-\theta)L_{\tilde{\mathbf{g}}}h(\mathbf{x})},\, \frac{-p(\mathbf{x})}{(1+\theta)L_{\tilde{\mathbf{g}}}h(\mathbf{x})}\right\}
\end{align}
The functions $p$ and $L_{\tilde{\mathbf{g}}}h$ are locally Lipschitz continuous, since $\mathbf{f}$, $\mathbf{g}$, $\eta$ and $h$ are locally Lipschitz continuous by assumption. Note that $L_{\tilde{\mathbf{g}}}h(\mathbf{x})\neq0$, because $L_{\tilde{\mathbf{g}}}h(\mathbf{x})$ is Lipschitz continuous, so there exists an $\epsilon>0$ and a ball around $\mathbf{x}$ such that $|L_{\tilde{\mathbf{g}}}h(\mathbf{x})|\geq\epsilon>0$. Thus, by Proposition~$1.30$ and Corollary~$1.31$ of~\cite{weaver2018lipschitz}, the ratio $\frac{-p(\mathbf{x})}{L_{\tilde{\mathbf{g}}}h(\mathbf{x})}$ is also locally Lipschitz continuous. The boundary function $u_l$ is locally Lipschitz continuous, because the point-wise maximum/minimum of two Lipschitz functions is also Lipschitz by Proposition $1.32$ of~\cite{weaver2018lipschitz}. Finally, the optimal solution of~\eqref{eq:sclarporb} can be written as $u^*(\mathbf{x}) = \max \{u_l(\mathbf{x}),u_0\}$. The baseline controller $u_0$ is assumed to be locally Lipschitz continuous. Thus, again using the Proposition~$1.32$ of~\cite{weaver2018lipschitz}, we have that $u^*$ is a locally Lipschitz continuous function of $\mathbf{x}\in D$. The case $L_{\tilde{\mathbf{g}}}h(\mathbf{x}) < 0$ can be handled similarly. 
\end{proof}

\subsection{Extensions}
\label{sec:ext}

\subsubsection{Sector-Bound for Individual Control Channels}
\label{sec:sectbndforindcontrolchannel}
Consider the case where the control input $\mathbf{u}(t) \in \R^m$ is vector valued ($m>1$). The sector bound corresponding to the constraint~\eqref{eq:sectorqc} allows for uncertainty that is coupled across input channels. An alternative model is to treat each input as having its own sector-bounded nonlinearity with no cross-coupling. In this case, the dynamics~\eqref{eq:mappedsys} can be written as:
\begin{align*}
    \dot{\mathbf{x}}(t) = \mathbf{f}(\mathbf{x}(t)) + \sum_{i=1}^{m} \tilde{\mathbf{g}}_i(\mathbf{x}(t)) \,(u_i(t) + w_i(t)), \hspace{0.1in} \mathbf{x}(0) = \mathbf{x}_0
\end{align*}
where $w_i(t)= \Delta_i(u_i(t),t)$, and each $\Delta_i$ satisfies the constraint $|w_i(t)| \leq \theta_i |u_i(t)|$, $\forall t \geq 0$. The worst-case uncertain input for an individual channel can be obtained as $w^*_i = -\theta_i |u_i| \mbox{sgn}(L_{\tilde{\mathbf{g}}_i}h(x))$. The linear programming trick (presented in Appendix~\ref{sec:linprogtrick} for a scalar input) can be used to separate positive and negative parts of the individual control input $u_i$. This yields a quadratic program for the safety filter.

\subsubsection{Robust Exponential CBF} 
\label{sec:RECBF}
The previous section presented robust control barrier functions with relative degree one, i.e. control input $\mathbf{u}$ shows up after differentiating the function $h$ once, which implies that $L_\mathbf{\tilde{g}}h(\mathbf{x}) \neq 0$ in Definition~\ref{def:RCBF}. However, in general $h$ can have a higher relative degree. Note that the dynamics in Equation~\eqref{eq:mappedsys} has the uncertain input $\mathbf{w}$ matched with the control input $\mathbf{u}$. This matching condition allows us to naturally extend the RCBF framework to higher relative degree robust CBFs (referred to as robust exponential CBFs or RECBF). The theory is similar to that of the (nominal) exponential CBF combined with the robustness argument already presented in this paper before. More details on (nominal) ECBF are provided in~\cite{nguyen2016exponential} with an overview in~\cite{ames2019control}. The example in Section~\ref{sec:example} demonstrates the RECBF design approach for the relative degree of two.

\subsubsection{Unifying with RCLF}
\label{sec:RCLF} 
Sometimes stability and safety objectives are in direct conflict~\cite{ames2019control}. If the baseline controller is not designed with robustness as a consideration, uncertainty may lead to unstable behavior before safety violation becomes an issue. A careful design should first consider providing a robust stability guarantee for the baseline controller using a robust control Lyapunov function (RCLF) (see Chapter 3 of~\cite{freeman2008robust}). The stability and safety objectives can also be combined in a single multi-objective optimization problem to design the controller $\mathbf{k(x)}$ in~\figref{fig:designic}. For nominal design ($\theta=0$) this problem is referred to as CLF-CBF QP~\cite{ames2019control}. A similar design can also be considered for the robust ($\theta>0$) counterpart, which yields a RCLF-RCBF SOCP problem. A specific implementation uses a hard constraint with a RCBF to enforce robust safety and a soft constraint with a RCLF (using a slack variable) to approximately enforce robust stability.

\subsubsection{Parametric Uncertainties}
\label{sec:param}
The robust design approach presented in this paper can also be extended to systems with uncertain parameters in the input function $\mathbf{g}(\mathbf{x})$. Let the system dynamics be given by:
\begin{align*}
    \mathbf{\dot{x}} = \mathbf{f}_0(\mathbf{x})
    + \left[\mathbf{g}_0(\mathbf{x}) + \sum_{i=1}^{n_p} \mathbf{g}_i(\mathbf{x})\delta_i\right]\mathbf{u}, \hspace{0.1in} \mathbf{x}(0) = \mathbf{x}_0
\end{align*}
where the functions $\mathbf{f}_0$ and $\mathbf{g}_0$ capture the nominal dynamics. The remaining terms capture the effects of uncertain, real parameters $\{ \delta_i \}_{i=1}^{n_p}$. Each parameter variation is assumed to be normalized such that $|\delta_{i}| \leq \theta_i$. 
Define $\mathbf{w}_i := \delta_i\mathbf{u}$ and rewrite the constraint on $\delta_i$ as an individual norm-bound constraint $\|\mathbf{w}_i\| \leq \theta_i \|\mathbf{u}\|_2$. Definition~\ref{def:RCBF} can be modified to consider inner optimization over each $\mathbf{w}_i$. The worst-case $\mathbf{w}^*_i$ for the inner optimization is then given by an expression similar to that of~\eqref{eq:wcw} using each $\theta_i$ and $L_{\tilde{\mathbf{g}}_i}h(\mathbf{x})$.

\begin{rem}
If $n_p=1$, $\mathbf{f_1(x)} = 0$ and $\mathbf{g_1(x)} = \mathbf{g_0(x)}$, then $\mathbf{\dot{x}} = \mathbf{f_0(x)} + \mathbf{g_0(x)\,(u + w)}$, with $\mathbf{w} = \delta_1\mathbf{u}$ and $|\delta_1 | \leq \theta_1$. These dynamics are similar to those in~\eqref{eq:mappedsys}.  This special case corresponds to a gain variation at the plant input as appears in the classical gain margin calculation.
\end{rem}

\section{Example: Vehicle Lateral Control}
\label{sec:example}


Consider a vehicle being driven on a straight road that must avoid a stationary obstacle. The lateral dynamics of vehicle are linearized at a constant longitudinal speed to obtain the following linear time-invariant (LTI) model~\cite{alleyne1997comparison}.
\begin{align}
   \dot{\mathbf{x}}(t) &= \mathbf{A}\mathbf{x}(t) + \mathbf{B}(u(t)+w(t)),\,\,\, \mathbf{x}(0) = \mathbf{x_0}\\
    |w(t)| &\leq \theta\,|u(t)| \nonumber
\end{align}
where $\mathbf{x}(t) = \bsmtx e(t) & \dot{e}(t) & \psi(t) & \dot{\psi}(t)\esmtx^\top\in\R^4$ is the linearized state and the control $u(t)\in \R$ is the front wheel steering angle input. The model from~\cite{alleyne1997comparison} is slightly modified to include the uncertain input $w(t)\in\R$, which represents nonlinearities (e.g. allowable saturation) and/or time-varying uncertainties. Here, $e$ is the lateral distance to the lane center and $\psi$ is the vehicle heading relative to the path. The longitudinal distance and velocity along the road are denoted by $s$ and $\dot{s} = 28\, m/s$ respectively. The longitudinal dynamics of the vehicle are not controlled. The state-space matrices and the vehicle parameters are given in~\cite{alleyne1997comparison} and are provided in Appendix~\ref{sec:vehicledynparam} for completeness. A baseline state-feedback controller was designed using linear quadratic regulator with cost matrices $Q=diag(10,1,\frac{1}{30},1)$ and $R=5$.
This was implemented to track the reference command $\mathbf{r}(t)\in\R^4$ as:
\begin{align*}
    u_0 = \mathbf{K}\cdot\mathbf{(r - x)},\,\,\,
    \mbox{where } \mathbf{K} =\bsmtx 1.41  &  0.41 & 3.30 & 0.24 \esmtx.
\end{align*}
This differs slightly from the feedback diagram in~\figref{fig:designic} due to
the inclusion of the reference command, i.e. the baseline
controller is of the form $u_0 = k_0(x, r)$. 

A stationary obstacle of radius $d$ is assumed at the origin. The safe set $\mathcal{C}$ is defined by Equation~\eqref{eq:safeset} with $h(\mathbf{x}) = e^2 + s^2 -d^2 \geq 0$ where $d$ is chosen based on geometries of the vehicle and the obstacle. The first and second time-derivatives of $h$ along a trajectory of state $\mathbf{x}$ are given by:
\begin{align*}
    \dot{h}(\mathbf{x}) = 2 e \dot{e} + 2 s \dot{s}, \,\,\,\,
    \ddot{h}(\mathbf{x}, u, w) = 2 e \ddot{e}(u,w) + 2\dot{e}^2 + 2\dot{s}^2
\end{align*}
Since the control input $u$ and uncertain input $w$ both appear in $\ddot{e}$, the system has relative degree two. This requires the robust exponential CBF as discussed in Section~\ref{sec:RECBF} to ensure safety. The set $\mathcal{U}\equiv\R$ is considered for simplicity. 


The initial state of the vehicle is $\mathbf{x}_0 = \bsmtx 2 & 0 & 0 & 0 \esmtx^\top$ with $s(0) = -20$. The safe distance $d$ is chosen as $3$ meters. The reference trajectory is selected to track the center of the lane, i.e. $\mathbf{r}(t) = \mathbf{0}\in\R^4$. The safety filter is designed using the RECBF-SOCP approach. The MATLAB implementation using the function~\texttt{coneprog} is available online at~\cite{github}.

Let the uncertainty level be $\theta = 0.5$ for the simulation study. \figref{fig:VLC_CBF_comparison} shows the simulation results for the uncertain plant with the worst-case uncertainty as in Equation (\ref{eq:wcw}). The shaded red circle represents the unsafe region. The ECBF and RECBF design poles are chosen to be two repeated poles at $-30$. Note that the nominal LQR controller runs into the obstacle due to its lack of safety consideration. The (nominal) ECBF-QP safety filter does not explicitly account for the uncertainty. Thus, its trajectory slightly violates the safety requirement around $s=0$ or $t\approx0.7$ seconds. The RECBF-SOCP trajectory avoids the obstacle successfully by choosing cautious control input to account for the uncertainty.
\begin{figure}[t]
    \centering
    \includegraphics[width=\linewidth]{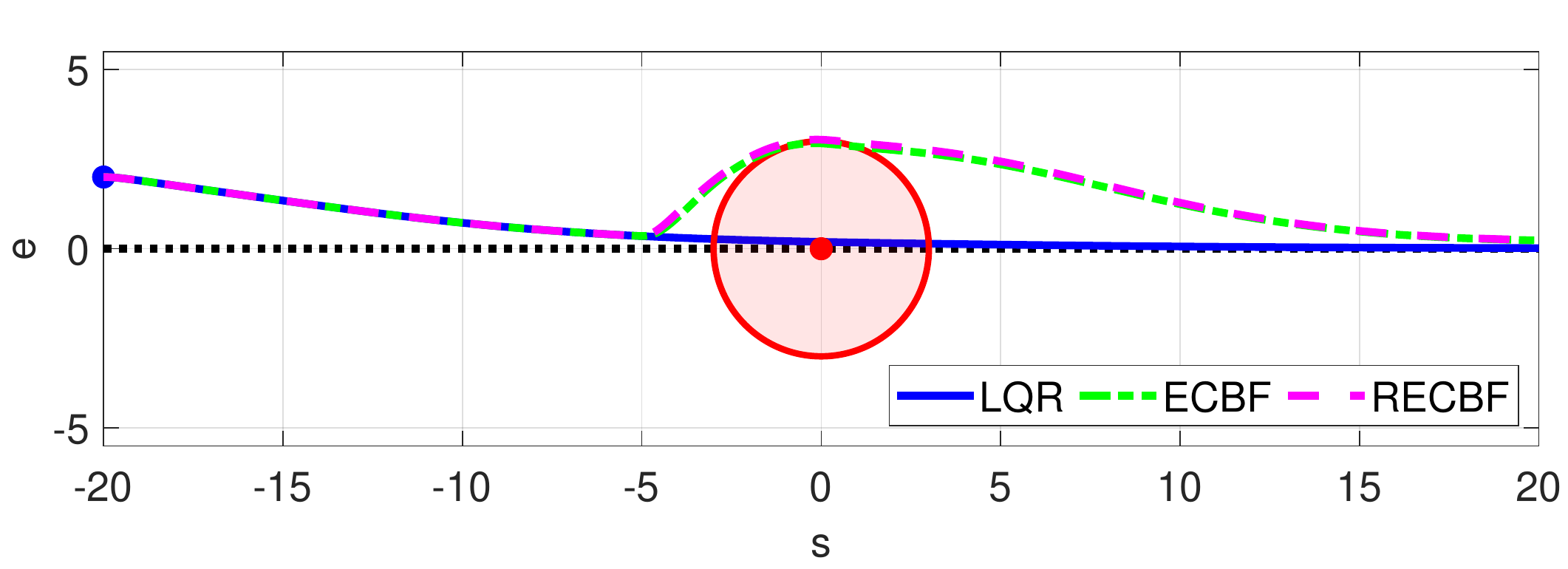}
    \includegraphics[width=\linewidth]{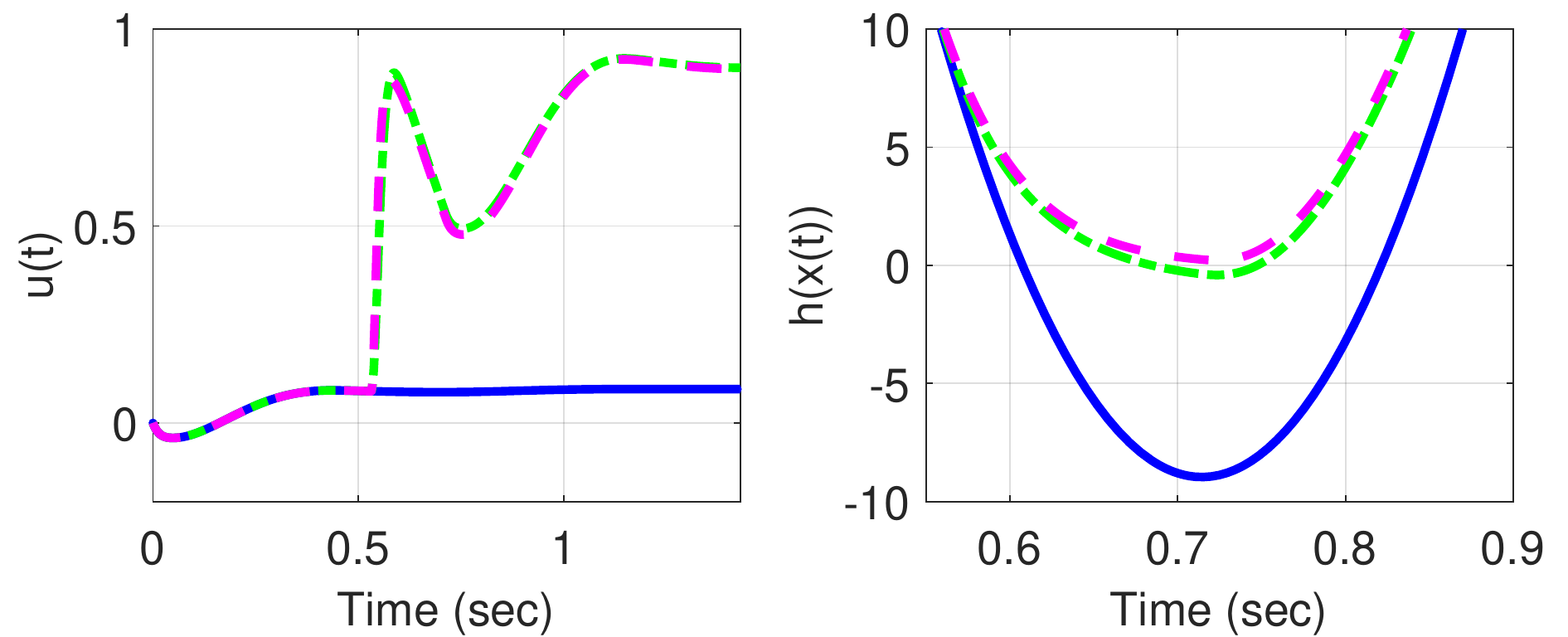}
    \caption{LQR, ECBF, and RECBF simulations with worst-case plant.}
    \label{fig:VLC_CBF_comparison}
\end{figure}

Next, the simulations are performed on the nominal plant. In other words the RECBF-SOCP controller is designed assuming $\theta>0$ but the simulations are performed on a nominal plant without the nonlinearity. \figref{fig:VLC_uncertainLevel} shows the RECBF-SOCP simulations with the same initial condition, but with the safety filter designed at different assumed values for the uncertainty bound $\theta$. Only the zoomed region around the obstacle is shown. It is observed that, as the model uncertainty level in design increases from $\theta = 0.2$ to $0.8$ the RECBF-SOCP generates more cautious trajectories around the obstacle. This is due to the fact that the proposed design explicitly consider the uncertainty at the plant input and yields robustly-safe control actions.
\begin{figure}[t]
    \centering
    \includegraphics[width=\linewidth]{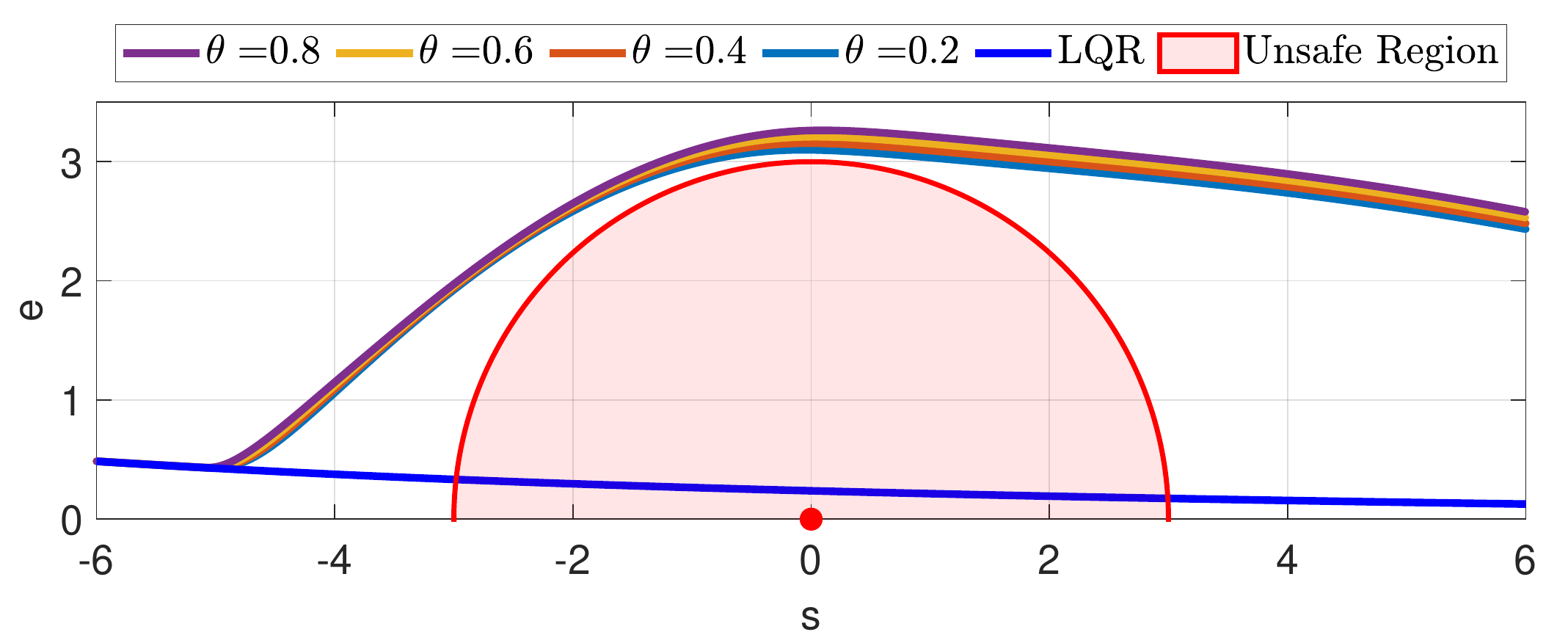}
    \caption{RECBF trajectories with varying uncertainty level $\theta$.}
    \label{fig:VLC_uncertainLevel}
\end{figure} 


\section{Conclusion}
\label{sec:conclusions}
This paper presents a robust control barrier function approach to handle sector-bounded uncertainties at the plant input. The proposed optimization problem can be written in terms of the second-order cone-program to be solved online. The robustness of the designed controller was studied in a lateral vehicle control example.

\section{Acknowledgments}
We thank Prof. Murat Arcak, Kate Schweidel and Adnane Saoud at the University of California, Berkeley for valuable discussions. We also thank Doug Philbrick at the NAWCWD China Lake for helpful comments.

	
\bibliographystyle{ieeetr}
\bibliography{RCBFMain}

\begin{appendices}
\section{Worst-Case Uncertain Input}
\label{sec:wcw}
Let $\lambda\geq 0$ be a Lagrange multiplier to rewrite the inner optimization problem~\eqref{eq:RCBF} as follows:
\begin{align*}
    \inf_{\mathbf{w}} \sup_{\lambda \geq 0} \big[  p(\mathbf{x}) + L_\mathbf{\tilde{g}}h(\mathbf{x}) (\mathbf{u} + \mathbf{w}) -\lambda (\theta^2 \mathbf{u}^\top \mathbf{u} - \mathbf{w}^\top \mathbf{w})\big] \geq 0 \nonumber
\end{align*}
Define the cost in bracket as $J(\mathbf{u},\mathbf{w},\lambda)$.
Strong duality holds because the objective and constraints are both convex~\cite{boyd2004convex}, thus $\sup_{\lambda \geq 0}\, \inf_{\mathbf{w}} J(\mathbf{u},\mathbf{w},\lambda) \geq 0$. Complete the square in $\mathbf{w}$ by adding and subtracting the term $\frac{1}{4\lambda} L_\mathbf{\tilde{g}}h(\mathbf{x})L_\mathbf{\tilde{g}}h(\mathbf{x})^\top$ in the cost function $J(\mathbf{u},\mathbf{w},\lambda)$. Then, minimizing over $\mathbf{w}$ yields the worst-case uncertain input as:
\begin{align}
    \label{eq:wcwl}
    \mathbf{w}^* = - \frac{1}{2\lambda^*} L_\mathbf{\tilde{g}}h(\mathbf{x})^\top
\end{align}
Plug in $\mathbf{w}^*$ in the cost function $J(\mathbf{u},\mathbf{w},\lambda)$ to obtain the cost function $J(\mathbf{u},\lambda)$ as follows:
\begin{align*}
    J(\mathbf{u},\lambda) := p(\mathbf{x}) &+ L_\mathbf{\tilde{g}}h(\mathbf{x}) \mathbf{u} -\lambda \theta^2 \mathbf{u}^\top \mathbf{u} - \frac{1}{4\lambda} L_\mathbf{\tilde{g}}h(\mathbf{x})L_\mathbf{\tilde{g}}h(\mathbf{x})^\top
\end{align*}
Use the first-order optimality condition to solve for optimal Lagrange multiplier $\lambda^*$ as follows:
\begin{align}
\lambda^* = \frac{\|L_\mathbf{\tilde{g}}h(\mathbf{x})\|_2}{2\theta\|\mathbf{u}\|_2}
\end{align} Plug $\lambda^*$ in~\eqref{eq:wcwl} to obtain the expression for $\mathbf{w}^*(\mathbf{u})$ as in~\eqref{eq:wcw}. 

\section{Linear Programming Trick}
\label{sec:linprogtrick}

This section demonstrates that the RCBF optimization in~\eqref{eq:RCBFOpt1} can be re-formulated as a linear program if the input is scalar. If $u_0$ satisfies the constraint in~\eqref{eq:RCBFOpt1} then $u^*=u_0$ and there is no need to run the optimization.  However, if $u_0$ is infeasible then:
\begin{align}
    \label{eq:infeasibleconstraint}
    p(\mathbf{x}) + L_{\tilde{\mathbf{g}}} h(\mathbf{x})\, u_0 - \theta |u_0| |L_{\tilde{\mathbf{g}}}h(\mathbf{x})| < 0
\end{align}
In this case, substitute $u = u_p - u_n$ and $|u| = u_p + u_n$ to rewrite the optimization problem~\eqref{eq:RCBFOpt1} as follows:
\begin{align}
\label{eq:linprogtrick}
   &\bsmtx u_p^*(\mathbf{x})\\ u_n^*(\mathbf{x}) \esmtx = arg \min_{u_p,u_n} \frac{1}{2} \|(u_p - u_n) - u_0\|^2_2 \\
    &\hspace{0.1in}\mbox{s.t. } p(\mathbf{x}) + L_\mathbf{\tilde{g}}h(\mathbf{x}) (u_p - u_n) - \theta (u_p + u_n) |L_\mathbf{\tilde{g}}h(\mathbf{x})| \geq 0\nonumber\\
  &\hspace{0.3in} u_p \geq 0,\, u_n \geq 0\nonumber
\end{align}
The reformulated problem has linear constraints. Hence this is a quadratic program with decision variables $u_p\geq 0$ and $u_n\geq 0$. This problem can be efficiently solved online instead of~\ref{eq:SOCP}. 

Next, we show that solutions of optimization problem~\eqref{eq:linprogtrick} are unique, i.e. $u_p$ and $u_n$ both can not simultaneously be greater than $0$, one of the variables must be zero for another to be nonzero. To show this by contradiction, let's assume that $u_p^* > 0$ and $u_n^* > 0$. There are $3$ cases to consider: 
\underline{Case-1:} If $u_p^* - u_n^* = u_0$ then the constraint with $(u_p^*,u_n^*)$ is 
\begin{align}
     0 &\leq p(\mathbf{x}) + L_\mathbf{\tilde{g}}h(\mathbf{x}) u_0 - \theta (u_p^* + u_n^*) |L_\mathbf{\tilde{g}}h(\mathbf{x})| \nonumber\\
    &\leq p(\mathbf{x}) + L_\mathbf{\tilde{g}}h(\mathbf{x}) u_0 - \theta |u_0| |L_\mathbf{\tilde{g}}h(\mathbf{x})|
\end{align}
This follows from $|u_0| \leq \max(u_p^* - u_n^*, u_n^* - u_p^*) \leq u_p^* + u_n^*$, which implies that $u_0$ is feasible. However, this case cannot happen as we do not run the optimization if $u_0$ is feasible.\\
\underline{Case-2:} Assume $u_p^* - u_n^* < u_0$. If $L_\mathbf{\tilde{g}}h(\mathbf{x}) > 0$ then define $\tilde{u}_n := u_n^* - \epsilon$ for $\epsilon > 0$. The assumption $u_n^* > 0$ implies that $\tilde{u}_n > 0$ for sufficiently small $\epsilon$ and the constraint remains satisfied with $(u_p^*, \tilde{u}_n)$ because,
\begin{align}
     &p(\mathbf{x}) + L_\mathbf{\tilde{g}}h(\mathbf{x}) (u_p^* - \tilde{u}_n) - \theta (u_p^* + \tilde{u}_n) |L_\mathbf{\tilde{g}}h(\mathbf{x})| \nonumber\\
&=\left[p(\mathbf{x}) + L_\mathbf{\tilde{g}}h(\mathbf{x}) (u_p^* - u_n^*) - \theta (u_p^* + u_n^*) |L_\mathbf{\tilde{g}}h(\mathbf{x})|\right]\ldots\nonumber\\ &\hspace{0.5in}+\epsilon\, (L_\mathbf{\tilde{g}}h(\mathbf{x}) + \theta|L_\mathbf{\tilde{g}}h(\mathbf{x})|) \geq 0
\end{align}
The term in brackets is $\geq 0$ by feasibility of $(u_p^*,u_n^*)$ and the term with $\epsilon$ is $\geq 0$ because $L_\mathbf{\tilde{g}}h(\mathbf{x}) > 0$. Thus, the pair $(u_p^*,\tilde{u}_n)$ is feasible and moreover it gives lower cost than $(u_p^*,u_n^*)$, because we have slightly increased $u_p^* - \tilde{u}_n$ toward $u_0$. If $L_\mathbf{\tilde{g}}h(\mathbf{x})<0$ then set $\tilde{u}_p = u_p^* + \epsilon$ and follow a similar argument to show that $(\tilde{u}_p,u_n^*)$ is feasible and gives lower cost than $(u_p^*,u_n^*)$. Thus, in either case ($L_\mathbf{\tilde{g}}h(\mathbf{x})<0$ or $>0$) we have that if $u_p^* > 0$ and $u_n^* > 0$ then it can not be an optimal point. Hence, by contradiction at least one of them must be zero.\\
\underline{Case-3:} Assume $u_p^*-u_n^*> u_0$. This case is similar to Case-$2$ and is not included.

\section{Vehicle Dynamics and Parameters}
\label{sec:vehicledynparam}

The state-space matrices are given by:
\begin{align*}
    \mathbf{A} &= \bmtx
     0 & 1 & 0 & 0  \\
     0 & \frac{C_{\alpha f}+C_{\alpha r}}{mU}  & -\frac{C_{\alpha f}+C_{\alpha r}}{m} & \frac{aC_{\alpha f}-bC_{\alpha r}}{mU}  \\
     0 & 0 & 0 & 1  \\
     0 & \frac{aC_{\alpha f}-bC_{\alpha r}}{I_zU} & \frac{aC_{\alpha f}-bC_{\alpha r}}{I_z} & \frac{a^2C_{\alpha f}+b^2C_{\alpha r}}{I_zU}   \\
    \emtx\\
    \mathbf{B} &= \bmtx
     0 \\ -\frac{C_{\alpha f}}{m} \\ 0 \\ -\frac{aC_{\alpha f}}{I_z}  
    \emtx 
\end{align*}

The vehicle parameters from~\cite{alleyne1997comparison} are given as follows:
\begin{table}[h]
    \label{tab:veh_para}
    \begin{tabular}{| l | l |l| }
        \hline
         $m$     & Vehicle mass                              & 1.67$\times 10^3$ kg  \\
         $I_z$   & Vehicle moment of inertia                 & 2.1$\times10^3$ kg-m$^2$ \\
         $a$     & Distance from vehicle CG to front axle    & 0.99 m    \\
         $b$     & Distance from vehicle CG to rear axle     & 1.7 m \\
         $U$     & Longitudinal velocity                     & 28 m/s \\
         $C_{\alpha f}$    & Front cornering stiffness     & -1.23$\times 10^5$ N/rad \\
         $C_{\alpha r}$    & Rear cornering stiffness      & -1.042$\times 10^5$ N/rad \\
         \hline
    \end{tabular}
\end{table}

\end{appendices}
\end{document}